\documentclass[11pt,reqno]{amsart}
\usepackage[margin=1.30in]{geometry}
\usepackage{setspace}

\usepackage[utf8]{inputenc}
\usepackage{amsmath,amsfonts,amssymb,amsopn,amscd,amsthm,bbm}
\usepackage{comment}
\usepackage{dsfont}
\usepackage{graphicx}
\usepackage{color}
\usepackage[colorlinks]{hyperref}
\usepackage{epigraph,todonotes}
\usepackage{enumerate}
\allowdisplaybreaks[4]

\DeclareMathOperator{\tr}{tr}
\DeclareMathOperator*{\argmax}{arg\,max}

\def\1{{\mathbf 1}}

\def\pa{\partial}

\def\T{\mathbb{T}}

\def\e{\epsilon}

\def\bs{\boldsymbol}

\def\N{{\mathbb N}}
\def\Z{{\mathbb Z}}

\def\R{{\mathbb R}}

\def\E{{\mathbb E}}

\def\Pc{{\mathcal P}}

\newcommand\ol{\overline}

\numberwithin{equation}{section}

\newtheorem{theorem}{Theorem}[section]
\newtheorem{proposition}{Proposition}[section]

\newtheorem{definition}{Definition}[section]

\newtheorem{lemma}{Lemma}[section]

\newtheorem{assumption}{Assumption}[section]

\newcommand{\ignore}[1]{}
\newcommand{\vertiii}[1]{{\left\vert\kern-0.25ex\left\vert\kern-0.25ex\left\vert #1 
    \right\vert\kern-0.25ex\right\vert\kern-0.25ex\right\vert}}

\title[Convergence Rate of Particle System]{Convergence Rate of Particle System for  Second-order  PDEs On Wasserstein Space} 
\author{Erhan Bayraktar}\thanks{E. Bayraktar is partially supported by the National Science Foundation under grant DMS-2106556 and by
the Susan M. Smith chair.}
\address{Department of Mathematics, University of Michigan}
\email{erhan@umich.edu}
\author{Ibrahim Ekren}\thanks{I. Ekren is supported in part by NSF Grant DMS 2007826.}
\address{Department of Mathematics, University of Michigan}
\email{iekren@umich.edu}
\author{Xin Zhang} 
\address{Department of Finance and Risk Engineering, New York University}
\email{xz1662@nyu.edu}


\begin{document}

\maketitle
\begin{abstract}
In this paper, we provide a convergence rate for particle approximations of a class of  second-order PDEs on Wasserstein space. We show that, up to some error term,  the infinite-dimensional inf(sup)-convolution of the finite-dimensional value function yields a super (sub)-viscosity solution to the PDEs on Wasserstein space. Hence, we obtain a convergence rate using a comparison principle of such PDEs on Wasserstein space. Our argument is purely analytic and relies on the regularity of value functions established in \cite{DaJaSe23}.
\end{abstract}

% REQUIRED
% \begin{keywords}
% Wasserstein space, second-order PDEs, viscosity solutions, comparison principle
% \end{keywords}

% % REQUIRED
% \begin{MSCcodes}
% 49L25, 60H30, 93E20
% \end{MSCcodes}

\section{Introduction}

In this paper, we consider particle approximations of the mean field PDE on the $d$-dimensional Torus $\T^d$ 
\begin{equation}\label{eq:HJB}
\begin{cases}
-\partial_t v(t,\mu)=\int_{\T^d} H(x,D_{\mu} v(t,\mu,x),\mu) \, \mu(dx) +\int_{\T^d}\tr (D_{x\mu}^2 v)(t,\mu,x) \, \mu(dx) \\
\quad \quad \quad \quad \quad \quad  +a\tr (\mathcal{H}v)(t,\mu), \\
\quad \   v(T,\mu)=G(\mu),
\end{cases}
\end{equation}
where $a$ is a nonnegative constant, $\mathcal{H}v(t,\mu)$ is the partial Hessian defined as in  \cite{BaEkZh23,DaJaSe23} 
\begin{align*}
\mathcal{H}v(t,\mu):= \int_{\T^d} D_{x\mu}^2 v(t,\mu,x) \, \mu(dx) + \int_{\T^d} \int_{\T^d} D_{\mu\mu}^2 v(t,\mu,x,y) \, \mu(dx) \mu(dy),
\end{align*}
and $H,G$ are respectively the Hamiltonian and the terminal condition
\begin{align*}
H=H(x,p,\mu): \T^d \times \R^d \times \Pc(\T^d) \to \R, \quad G=G(\mu):   \Pc(\T^d) \to \R. 
\end{align*}

Denote $\bs{x}=(x^1,\dotso,x^N) \in \T^{d N}$ and its empirical measure by $\mu^{\bs{x}}:=\frac{1}{N} \sum_{i=1}^N \delta_{x^i}$. The particle approximation of \eqref{eq:HJB} is provided by 
\begin{equation}\label{eq:nHJB}
\begin{cases}
 -\pa_t v^N(t,\bs{x})= \frac{1}{N} \sum_{i=1}^N H(x^i,ND_{x^i}v^N(t,\bs{x}),\mu^{\bs{x}})+\sum_{i=1}^N \Delta_{x^i}v^N(t,\bs{x}) \\
\quad \quad \quad \quad \quad \quad \ \ +a \sum_{i,j=1}^N \tr(D^2_{x^ix^j} v^N)(t,\bs{x}),  \\
\quad \  v^N(T,\bs{x})=G(\mu^{\bs{x}})
\end{cases}
\end{equation}
see \cite{DaJaSe23,MR4507678} and the references therein.
We provide a convergence rate of $v^N \to v$ based on the comparison principle and the regularity results obtained in \cite{DaJaSe23}.

It has been observed in \cite{BaEkZh23,DaJaSe23} that the partial Hessian term is equal to the second order derivative of $v$ in the barycenter of measures. More precisely, taking
\begin{align*}
V(t,z,\mu):=v(t,(I_d+z)_{\#}\mu), \quad (t,z,\mu) \in [0,T] \times \T^d \times \Pc(\T^d),
\end{align*}
we have that $ \nabla^2_z V(t,z,\mu)=\mathcal{H}v(t,(I_d+z)_{\#}\mu)$. We say $v$ is a viscosity solution to \eqref{eq:HJB} if $V$ is a viscosity solution to 
\begin{equation}\label{eq:eHJB}
\begin{cases}
-\partial_t V(t,z,\mu)=\int_{\T^d} H^e(x,z,D_{\mu} V(t,z,\mu,x),\mu) \, \mu(dx)  \\
\ \, \quad \quad\quad \quad \quad \ \  \ \  +\int_{\T^d}\tr (D_{x\mu}^2 V)(t,z,\mu,x) \, \mu(dx)+a\Delta_zV(t,z,\mu), \\
 \quad \  V(T,z,\mu)=G^e(z,\mu),
\end{cases}
\end{equation}
where 
\[H^e(x,z,p,\mu):=H(x+z,p,(I_d+z)_{\#}\mu), \quad G^e(z,\mu):=G((I_d+z)_{\#}\mu).\]

Similarly, define $V^N(t,z,\bs{x}):=v^N(t,z+\bs{x})$ where $z+\bs{x}:=(z+x^1,\dotso,z+x^N)$. Denoting the Fourrier-Wasserstein distance in \cite{BaEkZh23,DaJaSe23,soner2023viscosity} by $\rho_*$, we show that the inf-convolution of $V^N$ with respect to $\rho_*$ given by
{\small\[\ol{V}^{N,\epsilon}(t,z,\mu):= \inf_{(s,w, \bs{ x}) \in [0,T] \times \T^{d+dN} } \left( V^N(s,w,\bs{x}) +\frac{1}{2\e} |t-s|^2+\frac{1}{2\e} |z-w|^2+ \frac{1}{2\epsilon} \rho^2_*(\mu^{\bs{x}}, \mu) \right). \]}
is a viscosity supersolution to \eqref{eq:eHJB} up to some error term $E(\e,N)$, and therefore by comparison for all $(t,z,\bs{x}) \in [0,T] \times \T^{d+dN}$ we have
\[ V^N(t,z,\bs{x})+E(\e,N)  \geq \ol{V}^{N,\e}(t,z,\mu^{\bs{x}})+E(\e,N) \geq V(t,z,\mu^{\bs{x}}).\] By the same token, we obtain the other direction
\[V^N(t,z,\bs{x})-E(\e,N)   \leq V(t,z,\mu^{\bs{x}}), \quad (t,z,\bs{x}) \in [0,T] \times \T^{d+dN}, \]
and hence 
\[|V^N(t,z,\bs{x})-V(t,z,\mu^{\bs{x}})| \leq E(\e,N).\]
Then choosing $\e$ properly as a function of $N$, one obtains that $E(\e,N)\approx C\alpha^{1/3}(N)$ where $C$ is a positive constant independent of $N$ and 
\begin{align*}
\alpha(N)=
\begin{cases}
N^{-1/2}, \quad & \text{if $d=1$}, \\
N^{-1/2} \log(N), & \text{if $d=2$}, \\
N^{-1/d}, & \text{if $d>2$}. 
\end{cases}
\end{align*}

Let us denote by $B_R \subset \R^d$ the closed ball at the origin with radius $R$ and $k_*=\lfloor d/2 \rfloor+3$. Using the definitions of functional spaces in \cite{DaJaSe23}, we make the following assumptions on the coefficients $H,G$. 
\begin{assumption}\label{assume} 
\begin{enumerate}
\item[(i)] $G$ is $k_*$-times continuously differentiable and Lipschitz with respect to $C^{-k_*}$; 
\item[(ii)]  $H$ is $k_*$-times continuously differentiable in all variables, and satisfies the regularity condition, with $C_H>0$ 
\[ \left| H(x,p,\mu)-H(x',p',\mu') \right| \leq C_H(1+|p|+|p'|)(|x-x'|+|p-p'|+W_1(\mu,\mu') )\]
for any $x,x' \in \T^d$, $p,p' \in \R^d$, $\mu,\mu' \in \mathcal{P}(\T^d)$.
For each $R >0$, there is a constant $C_{R}$ such that for each $(x,p) \in \T^d \times B_R$
\begin{align*}
|H(x,p,\mu)-H(x,p,\mu')| \leq C_R \lVert \mu -\mu' \rVert_{C^{-k_*}}, 
\end{align*}
and for each $\mu \in \Pc(\T^d)$, 
$$\sup_{ \mu \in \Pc(\T^d)} \lVert H(\cdot,\cdot, \mu) \rVert_{C^{k_*}(\T^d \times B_R)} \leq C_R .$$\end{enumerate}
\end{assumption}

\begin{theorem}\label{thm:main}
Under Assumption~\ref{assume}, $v^N$ converges to $v$ with the rate $\alpha^{1/3}(N)$
\[ \sup_{(t,\bs{x}) \in [0,T] \times \T^{dN}}\left| v^N(t, \bs{x})-v(t, \mu^{\bs{x}}) \right| \leq C \alpha^{1/3}(N), \]
where $C$ is a constant independent of $N$.
\end{theorem}

Our result provides the first convergence rate for particle approximations of partially second-order PDEs on the Wasserstein space. Such equations arise in mean-field control problems when the control of common noise is state-independent, as well as in stochastic control with partial observation; see, e.g., \cite{BaEkZh23}. In the former case, the particle approximation justifies the mean-field formulation of controlling a large population of agents; see, e.g., \cite{CaLe18, MR3752669} and references therein.

Without the common noise, i.e. setting $a=0$ in \eqref{eq:HJB}, \cite{2023arXiv230508423D} obtains the convergence rate of $1/\sqrt{N}$ under some convexity assumptions. The argument relies on the fact that super-convolution of semi-concave functions provides desired regularity. While \cite{2023arXiv230508423D} shows better convergence rate,  our argument is purely analytic and doesn't rely on the semi-concavity/convexity of $H$. We would like to mention that the argument of our result makes use of the regularity of $v^N$ established in \cite{DaJaSe23} where only the convergence  $$ \sup_{(t,\bs{x}) \in [0,T] \times \T^{dN}}\left| v^N(t, \bs{x})-v(t, \mu^{\bs{x}}) \right| \to 0$$ is provided but not its convergence rate.   Moreover, using the same method of the infinite dimensional inf/sup-convolution and the comparison result from \cite{BaEkZh23}, one should also be able to show the convergence rate for the same type of equations on $\R^d$. But to illustrate the main idea of the argument, we choose to work on $\T^d$ to avoid the technical issue of non-compactness.

The remainder of the paper is organized as follows. We will discuss some related literature in Section 1.1, and introduce notations in Section 1.2. In Section 2, we will present the definition of viscosity of solution and some preliminary results. The main result will be proved in Section 3.

\subsection{Related Literature}

PDEs on Wasserstein space appear in mean field games and McKean-Vlasov control problems \cite{cdll2019,MR3752669,MR3753660,MR4499277, MR3739204, cosso2021master}, and also in filtering problems \cite{martini2023kolmogorov,gozzi2000hamilton,bandini2019randomized,BaEkZh23,JMLR:v24:22-1001}. Various notions of differentiability for functions on Wasserstein space have been defined, and in this paper we adopt the one introduced by Lions in \cite{cdll2019}. It is stronger than the geometric definition of differentiability, and allows a version of It\^{o}'s formula which is crucial for control problems.

The comparison principle of PDEs on Wasserstein space has attracted lots of attention. Viscosity solutions of first-order PDEs on the Wasserstein space have been studied in \cite{MR4604196,soner2023viscosity,SoYa24,burzoni2020viscosity,MR4595996, cosso2021master,2023arXiv230604283B,2023arXiv230815174D}. It is worth noting that the Fourier-Wasserstein metric $\rho_*$ was first used in the study of viscosity solution by \cite{SoYa24}. The comparison principle of partially second-order equations, in which the second-order derivative in measure appears in the form of partial Hessian, have been studied in \cite{BaEkZh23,DaJaSe23,gangbo2021finite,2023arXiv230604283B}. Let us mention that \cite{gangbo2021finite,2023arXiv230604283B} adopted different notions of differentiability on Wasserstein space. Fully second-order PDEs on the Wasserstein space are related to measure-valued martingale optimization problems. \cite{cox2021controlled} proves a uniqueness result for equations that are exact limit of finite dimensional approximations. PDEs on the Wasserstein space also appear in mean-field optimal stopping problems \cite{MR4604196,MR4613226,2023arXiv230709278P}. %In addition, viscosity solution of path-dependent PDEs were investigated in \cite{ekren2014viscosity,ekren2016viscosity,ren2017comparison}. {\color{red} last sentences to be deleted}

Convergence of particle system in mean field control problems were studied in \cite{2022arXiv221016004T,2022arXiv221100719T,MR4595996,Ce21} based on viscosity theory, while \cite{doi:10.1137/16M1095895} provided a probabilistic argument. The convergence rate for first-order PDEs on Wasserstein space was obtained in \cite{2023arXiv231211373C,2023arXiv230508423D,2022arXiv220314554C}. Assuming the existence of smooth solution to mean-field PDEs, \cite{MR4507678} got the optimal convergence rate by a verification argument. 

\subsection{Notations} 
Define $\T^d=\R^d/(2 \pi \mathbb Z)^d$, and take Fourier basis
\[e_l(x):=(2\pi)^{-d/2} e^{i l \cdot x}, \quad x \in \T^d, \, k \in \mathbb Z^d. \]
For any complex number $z \in \mathbb C$, we denote its complex conjugate by $z^*$. 
For any $f \in L^2(\T^d)$, we define $F_l(f):= \int_{\T^d} f(x) e_l^*(x) \,dx$. For any $k \in \mathbb N$, we define,
\begin{align*}
\lVert f \rVert_{k}^2:= \sum_{l \in \Z^d} (1+|l|^2)^k |F_l(f)|^2 
\end{align*}
and the Sobolev space
\[\mathbb H^{k}(\T^d)=\left\{f \in L^2(\T^d): \, \lVert f \rVert_{k} < \infty \right\}.\]
The space of $k$-th continuously differentiable function is denoted by $C^k(\T^d)$ with the norm defined as $$\lVert f \rVert_{C^k}=\sum_{j\in \N^d:|j|\leq k} \lVert D^j f \rVert_{L^\infty}.$$
For any signed Borel measure $\eta$ on $\T^d$, we define 
\begin{align*}
\lVert \eta \rVert_{-k}= \sup_{ \lVert f \rVert_{k}  \leq 1 } \int_{\T^d} f (x) \, \eta(dx), 
\end{align*}
\begin{align*}
F_l(\eta)= \int_{\T^d} e_l(x) \, d \eta , \quad l \in \mathbb Z^d, 
\end{align*}
and also 
\begin{align*}
\lVert \eta \rVert_{C^{-k}}= \sup_{ \lVert f \rVert_{C^k}  \leq 1 } \int_{\T^d} f (x) \, \eta(dx).
\end{align*}
Then for any $\mu,\nu \in \Pc(\T^d)$, we define $\rho_{-k}(\mu,\nu)=\lVert \mu-\nu \rVert_{-k}$. Throughout the paper, we take $k_*=\lfloor d/2 \rfloor+3$ and denote $\rho_*=\rho_{-k_*}$. 

Throughout the paper, we adopt Lions differentiability for functions defined on $\Pc(\T^d)$; see e.g. \cite[Chapter 5]{MR3752669}. 
\section{Definition and preliminaries}

\subsection{Definition of viscosity solution}
First we introduce the notion of viscosity solution from \cite{DaJaSe23}. Suppose $v:[0,T] \times \Pc(\T^d) \to \R$ is a smooth solution to \eqref{eq:HJB}. Define $V(t, z ,\mu): [0,T] \times\T^d \times  \Pc(\T^d) \to \R$ via 
\begin{align}\label{eq:extension}
V(t,z,\mu):=v(t,(I_d+z)_{\#}\mu). 
\end{align}
It can be easily verified that 
\begin{align*}
D_\mu V(t,z,\mu)(x) &= D_{\mu} v(t,(I_d+z)_{\#} \mu) (x+z), \\
D_{x\mu}^2 V(t,z,\mu)(x) &= D^2_{x\mu} v(t,(I_d+z)_{\#} \mu) (x+z), \\
\Delta_z V(t,z,\mu)&=\tr(\mathcal{H}v)(t,(I_d+z)_{\#} \mu),
\end{align*}
and $V$ satisfies \eqref{eq:eHJB}. 

We say $v$ is a viscosity solution of \eqref{eq:HJB} if $V$ is a viscosity solution of \eqref{eq:eHJB}. More precisely, we have the set of test functions. 

\begin{definition}
Denote by $C_{\mathfrak{p}}^{1,2,2}([0,T] \times \T^d \times \Pc(\T^d))$ the set of continuous functions $\Phi(t,z,\mu): [0,T] \times \T^d \times \Pc(\T^d) \to \R$ such that the derivatives 
\begin{align*}
(\pa_t \Phi, D_z \Phi,D^2_{zz} \Phi)(t,z,\mu): [0,T] \times \T^d \times \Pc(\T^d) \to \R \times \R^d \times \R^{d\times d}
\end{align*}
as well as 
\begin{align*}
(D_{\mu} \Phi,D^2_{x\mu}\Phi)(t,z,\mu,x): [0,T] \times \T^d \times \Pc(\T^d) \times \T^d \to \R^d \times \R^{d\times d}\end{align*}
exist and are continuous. 
\end{definition}

\begin{definition}
An upper semi-continuous function $v: [0,T] \times \Pc(\T^d) \to \R$ is called a viscosity subsolution to \eqref{eq:HJB} if its extension $V$ via \eqref{eq:extension} is a viscosity subsolution to \eqref{eq:eHJB}, i.e., $V(T,z,\mu) \geq G^e(z,\mu)$, and for any $\Phi \in C_{\mathfrak{p}}^{1,2,2}([0,T] \times \T^d \times \Pc(\T^d))$ such that $V-\Phi$ obtains a local maximum at $(t_0,z_0,\mu_0)\in [0,T) \times \T^d \times \Pc(\T^d)$, we have 
\begin{align*}
-\partial_t \Phi(t_0,z_0,\mu_0)\leq &\int_{\T^d} H^e(x,z_0,D_{\mu} \Phi(t_0,z_0,\mu_0,x),\mu) \, \mu_0(dx) \\ &+\int_{\T^d}\tr (D_{x\mu}^2 \Phi)(t_0,z_0,\mu_0,x) \, \mu_0(dx) +a\Delta_z \Phi(t_0,z_0,\mu_0), 
\end{align*}
Similarly, we define viscosity supersolution. A continuous function $$v:[0,T] \times \Pc(\T^d) \to \R$$ is called a viscosity solution if its extension $V$ via \eqref{eq:extension} is a viscosity subsolution and supersolution to \eqref{eq:eHJB} at the same time. 
\end{definition}

\subsection{Preliminary results}

As the HJB equation \eqref{eq:nHJB} of particle system is uniformly elliptic, one can show the regularity of solution. The following result is from \cite[Lemma 3.1, Theorem 3.2]{DaJaSe23}.
\begin{lemma}\label{lem:regular}
Under Assumption~\ref{assume}, there exists a unique classical solution $v^N$ to \eqref{eq:nHJB} and positive constants $C$ independent of $N$ such that for any $N \in \mathbb{N}$, $i \in \{1,2,\dotso, N\}$, $k \leq k_*$, 
\[\left| D_{x^i}^k v^N(t,\bs{x}) \right|\leq \frac{C}{N}, \quad \forall \, (t,\bs{x}) \in [0,T] \times \mathbb \T^{d N}. \]
In addition, for all $0 \leq s <t \leq T$ and $\bs{x},\bs{y} \in \mathbb{T}^{d  N}$, we have 
\[ \left| v^N(t,\bs{x})-v^N(s,\bs{y}) \right| \leq C \left(\sqrt{t-s}+W_1(\mu^{\bs{x}},\mu^{\bs{y}}) \right). \]  
\end{lemma}

Let us provide an approximation of $v^N$ that has been used in \cite{DaJaSe23,cosso2021master}. Given any $\mu \in \mathcal{P}(\mathbb{T}^d)$, denote by $\mu^{\otimes N}$ the $N$ fold product of probability measure $\mu$. Let us define $\hat{v}^N:[0,T] \times \mathcal{P}(\mathbb{T}^d)$ via 
\[\hat{v}^N(t,\mu)=\int_{\mathbb{T}^{d  N}} v^N(t,\bs{y}) \, \mu^{\otimes N}(d \bs{y}). \]

\begin{lemma}\label{lem:finiteapprox}
Under Assumption~\ref{assume}, the following inequality holds for any $\bs{x},\bs{y} \in \mathbb T^{d N}$ with a positive constant $C$ independent of $N$, 
\[ \left|  \hat{v}^N(t, \mu) - \hat{v}^N(t,\nu) \right| \leq C \rho_*(\mu,\nu).\]
\end{lemma}
\begin{proof}
It can be verified that $\hat{v}^N$ is linearly differentiable, and thanks to the symmetry its derivative  is given by {\small
\begin{align*}
\frac{\delta \hat{v}^N}{\delta \mu}(t, \mu, x)= \sum_{i=1}^N \int_{\mathbb{T}^{d  (N-1)} } v^N(t, y^1,\dotso,y^{i-1},x, y^i,\dotso,y^{N-1}) \, \mu^{\otimes (N-1)}(dy^1,\dotso,dy^{N-1}) 
\end{align*}}
In light of Lemma~\ref{lem:regular}, we get an estimate of the $\mathbb H^{k_*}$ norm 
\begin{align*}
\sup_{(t,\mu) \in {[0,T] \times \T^d}} \lVert \frac{\delta \hat{v}^N}{\delta \mu}(t, \mu, \cdot) \rVert_{{k_*}} \leq \sum_{k=1,\dotso,k_*} \sum_{i=1}^N \lVert D_{x^i}^k v^N\rVert_{\infty} \leq C.
\end{align*}
Therefore, by the definition of linear derivative, one immediately concludes that 
\begin{align*}
\left|  \hat{v}^N(t, \mu) - \hat{v}^N(t,\nu) \right| \leq  C \lVert \mu -\nu \rVert_{-k_*} =C \rho_*(\mu,\nu).
\end{align*}

\end{proof}

The next proposition shows that $\hat v^N$ is indeed close to $v^N$ and thus almost Lipschitz. 

\begin{proposition}\label{prop:Lipschitz}
For any $\bs{x},\bs{y} \in \mathbb{T}^{d N}$, the following estimates of the value function holds 
\[|v^N(t,\bs{x})-\hat{v}^N(t,\mu^{\bs{x}})| \leq C\alpha(N) \]
where $C$ is a positive constant independent of $N$. Together with Lemma~\ref{lem:finiteapprox}, immediately we obtain that 
$$|v^N(t,\bs{x})-v^N(t,\bs{y})|\leq  C\left( \rho_*(\mu^{\bs{x}},\mu^{\bs{y}})+\alpha(N) \right).$$
\end{proposition}
\begin{proof}
Let us prove $\left| v^N(t,\bs{x})-\hat{v}^N(t,\mu^{\bs{x}}) \right| \leq C \alpha(N)$. Indeed, according to the definition of $\hat{v}^N$ and Lemma~\ref{lem:regular} 
\begin{align*}
\left| v^N(t,\bs{x})-\hat{v}^N(t,\mu^{\bs{x}}) \right| & \leq  \int_{\mathbb{R}^{d \times N}} \left|v(t,\bs{x})-v(t,\bs{y}) \right| \, (\mu^{\bs{x}})^{\otimes N}(d\bs{y})\\
& \leq C \int_{\mathbb{R}^{d \times N}}  W_1(\mu^{\bs{x}}, \mu^{\bs{y}}) \, (\mu^{\bs{x}})^{\otimes N}(d\bs{y})=C\E\left[W_1(\mu^{\bs{x}}, \hat{\mu}^{\bs{x}}) \right],
\end{align*}
where $\hat{\mu}^{\bs{x}}$ denotes the empirical measure of $\mu^{\bs{x}}$ with $N$ samples. According to \cite{MR3383341}, it is bounded from above by $\alpha(N)$. 
\end{proof}

By Arzel\`{a}–Ascoli theorem, Lemma~\ref{lem:finiteapprox} implies there exists a  subsequence of $\hat{v}^N$ uniformly converging to a function $v$, which is also $\rho_*$-Lipschitz. Moreover according to Proposition~\ref{prop:Lipschitz}, $v$ is also a limiting point of $v^N$, and hence is a viscosity solution to \eqref{eq:HJB} by a standard argument. Finally, the uniqueness of viscosity solution has been proved in \cite{DaJaSe23} adapting the techniques in \cite{BaEkZh23}. 

\begin{proposition}\label{prop:existence}
Under Assumption~\ref{assume} for \eqref{eq:HJB}, any upper semi-continuous subsolution to is smaller than lower semi-continuous supersolution. Then there is a unique  viscosity solution $v$ to \eqref{eq:HJB} on $\mathcal{P}(\mathbb{T}^d)$.
\end{proposition}

At the end of section, we provide a simple lemma quantifying the denseness of empirical measures in $\Pc(\T^d)$. 
\begin{lemma}\label{lem:dense}
For any $\mu \in \mathcal{P}(\mathbb{T}^d)$, we have $$\inf_{\bs{x} \in \mathbb{T}^{d  N}}\rho_*(\mu^{\bs{x}},\mu) \leq C \alpha(N).$$
\end{lemma}
\begin{proof}
Thanks to Sobolev embedding theorem \cite[Corollary 9.13]{MR2759829}, as $k_*=\lfloor d/2 \rfloor+2$ any $f \in \mathbb{H}^{k_*}(\T^d)$ is Lipschitz with coefficient proportional to $\lVert f \rVert_{{k_*}}$. Therefore, it can be easily seen that 
\begin{align*}
\rho_*(\mu,\nu)\leq C W_1(\mu,\nu),
\end{align*}
with a universal constant $C$. Let us consider an \emph{i.i.d.} sequence $X_n$ with distribution $\mu$, and define $\hat{\mu}^N=\frac{1}{N} \delta_{X_n}$ to be the empirical measure of $\mu$. With the sample complexity of $W_1$, see e.g. \cite{MR3383341}, we deduce that
\begin{align*}
\E[\rho_*(\hat{\mu}^N,\mu)]\leq C\alpha(N),
\end{align*}
and hence 
\begin{align*}
\inf_{\bs{x} \in \mathbb{T}^{dN}}\rho_*(\mu^{\bs{x}},\mu) \leq C \alpha(N). 
\end{align*}
\end{proof}

\section{Proof of Theorem~\ref{thm:main}} 

We only prove the inequality 
\begin{align}\label{eq:finequality}
V^N(t,z,\bs{x}) \geq V(t,z,\mu^{\bs{x}})-C \alpha^{1/3}(N),
\end{align}
and the proof for the other direction 
\[V^N(t,z,\bs{x}) \leq V(t,z,\mu^{\bs{x}})+C \alpha^{1/3}(N) \] 
is similar. The argument is based on the fact  that the inf convolution preserves the property of being a supersolution, and will be divided into several steps. 

For any $z \in \T^d, \bs x \in \T^{dN}$, denote $z+\bs{x}:=(z+x^1,\dotso, z+x^N)$. Recall that $v^N$ denotes the classical solution to \eqref{eq:nHJB}. Introducing
\[ V^N(t,z,\bs{x}):= v^N(t,z+\bs{x}) ,\]
let us consider the inf-convolution of $V^N${\small
\begin{align*}
\ol{V}^{N,\epsilon}(t,z,\mu):= \inf_{(s,w, \bs{ x}) \in [0,T] \times \T^{d+dN} } \left( V^N(s,w,\bs{x}) +\frac{1}{2\e} |t-s|^2+\frac{1}{2\e} |z-w|^2+ \frac{1}{2\epsilon} \rho^2_*(\mu^{\bs{x}}, \mu) \right).
\end{align*}}It is clear that $V^N(t,z,\bs{x}) \geq \ol{V}^{N,\e}(t, z,\mu^{\bs{x}})$ for every $x \in \mathbb{T}^{d N}$. We prove that $\ol{V}^{N,\e}$ is a  viscosity supersolution to \eqref{eq:HJB} up to some error.

Suppose $\Phi:[0,T] \times \T^d \times \mathcal{P}(\mathbb{T}^d) \to  \R$ is a regular test function in $C_{\mathfrak{p}}^{1,2,2}([0,T] \times \T^d \times \Pc(\T^d))$ and we have a local strict minimum $(t_0, z_0,\mu_0)$ at $\ol{V}^{N,\e} -\Phi$ with $t_0 < T$, and we aim to show that 
\begin{align}\label{eq:supersol}
-\pa_t \Phi(t_0,z_0,\mu_0)+E(N,\e) \geq& \int_{\T^d} H^e(x,z_0,D_{\mu} \Phi(t_0,z_0,\mu_0,x),\mu_0) \, \mu_0(dx)  \notag \\
&+\int_{\T^d}\tr (D_{x\mu}^2 \Phi)(t_0,z_0,\mu_0,x) \, \mu_0(dx)    +a\Delta_z \Phi(t_0,z_0,\mu_0), 
\end{align}
where $E(N,\e)$ is an error term to be determined later.

By compactness, take $(s_0,w_0,\bs{x}_0)$ such that 
\begin{align*}
\ol{V}^{N,\e}(t_0, z_0, \mu_0)= V^N(s_0,w_0,\bs{x}_0)+\frac{1}{2\e}|s_0-t_0|^2+\frac{1}{2\e} |w_0-z_0|^2+ \frac{1}{2\e}\rho_*^2(\mu^{\bs{x}_0},\mu_0). 
\end{align*}
Then we have the following crucial inequality, for every  $(s,w,\bs{x})$ and $(t,z,\mu)$, 
\begin{align}\label{eq:key}
&V^N(s,w,\bs{x})+\frac{1}{2\e} \left(|s-t|^2+ |w-z|^2+\rho_*^2(\mu^{\bs{x}},\mu) \right)-\Phi(t,z,\mu)  \\
& \geq \ol{V}^{N,\e}(t,z,\mu)-\Phi(t,z,\mu) \notag \\
&\geq  \ol{V}^{N,\e}(t_0,z_0,\mu_0)-\Phi(t_0,z_0,\mu_0) \notag \\
&=V^N(s_0,w_0,\bs{x}_0)+\frac{1}{2\e} \left(|s_0-t_0|^2+ |w_0-z_0|^2+ \rho_*^2(\mu^{\bs{x}_0},\mu_0) \right)-\Phi(t_0,z_0,\mu_0). \notag
\end{align}

To make use of the viscosity property of $V^N$, let us define a finite dimensional test function $\Phi^{\e}$ by super-convolution,
\begin{align}\label{eq:defphie}
\Phi^\e(s,w,\bs{x}):&=\sup_{z \in \T^d }\left\{\Phi(t_0,z,\mu_0)-\frac{1}{2\e} |w-z|^2\right\}-\frac{1}{2\e}|s-t_0|^2-\frac{1}{2\e}\rho_*^2(\mu^{\bs{x}},\mu_0).
\end{align}
 Then inequality \eqref{eq:key} implies that 
\begin{align}\label{eq:Nmaxmia}
V^N(s,w,\bs{x})-\Phi^\e(s,w,\bs{x})& \geq \ol{V}^{N,\e}(t_0,z_0,\mu_0)-\Phi(t_0,z_0,\mu_0) \notag \\
&\geq V^N(t_0,w_0,\bs{x}_0)-\Phi^\e(s_0,w_0,\bs{x}_0). 
\end{align}
Therefore, $V^N - \Phi^\e$ obtains a local minimum at $(t_0,w_0,\bs{x}_0)$, and we are going to invoke the viscosity property of $V^N$. The next lemma provides derivatives of $\Phi^\e$. 

\begin{lemma}\label{lem:derivatives}
We have the following equalities, for every $w \in \T^d$ and every $\bs{x} \in \T^{dN}$
\begin{align*}
\pa_s \Phi^{\e}(s_0,w,\bs{x})&=\frac{1}{\e}(t_0-s_0),\\
N D_{x^i}\Phi^{\e}(s_0,w,\bs{x})&=\frac{1}{2\e } D_{ \nu} \rho_*^2(\mu^{\bs{x}},\mu_0)(x^i), \quad \quad  \quad \quad \hspace{86pt}\ \,i=1,\dotso, N,\\
 N D_{x^i}^2 \Phi^{\e}(s_0,w,\bs{x})&=\frac{1}{2\e } D^2_{x \nu} \rho_*^2(\mu^{\bs{x}},\mu_0)(x^i)+\frac{1}{2\e N} D^2_{\nu} \rho_*^2(\mu^{\bs{x}_0},\mu_0)(x^i,x^i), \, \,  i=1,\dotso, N, 
\end{align*}
and at $(t_0,z_0,\mu_0)$
\begin{align*}
\pa_t \Phi(t_0,z_0,\mu_0) &=\frac{1}{\e}(t_0-s_0),\\
D_{\mu}\Phi(t_0,z_0,\mu_0)(x^i)& =\frac{1}{2\e } D_{ \nu} \rho_*^2(\mu^{\bs{x}_0},\mu_0)(x^i), \quad \quad  \quad \quad \ i=1,\dotso, N,\\
D_{x\mu}^2 \Phi(t_0,z_0,\mu_0)(x^i) &=\frac{1}{2\e } D^2_{x \nu} \rho_*^2(\mu^{\bs{x}_0},\mu_0)(x^i), \quad \quad \quad \quad  i=1,\dotso, N. 
\end{align*}
Moreover, $w \mapsto \Phi^{\e}(s_0,w,\bs{x}_0)$ is second order differentiable \emph{a.e.} At every $w \in \T^d$ where it is second order differentiable, we have 
\begin{align*}
\Delta_w \Phi^{\e}(s_0,w,\bs{x}_0) & \geq  \Delta_z \Phi(t_0,z(w),\mu_0),
\end{align*}
where 
\[z(w) \in \argmax_{z \in \T^d} \left\{\Phi(t_0,z,\mu_0)-\frac{1}{2\e}|w-z|^2 \right\}. \]
\end{lemma}
\begin{proof}
Taking derivative with respect to $s$ in \eqref{eq:defphie},
\begin{align*}
\pa_s \Phi^{\e}(s_0,w,\bs{x})=\frac{1}{\e} (t_0-s_0). 
\end{align*}
Inequality \eqref{eq:key} implies that $t_0$ is a maximizer of 
\[\mu \mapsto \Phi(t,z_0,\mu)-\frac{1}{2\e}|s_0-t|^2,\]
and thus 
\begin{align*}
\pa_t \Phi(t_0,z_0,\mu_0)= \frac{1}{\e}(t_0-s_0)=\pa_s \Phi^{\e}(s_0,w,\bs{x}_0).
\end{align*}

Taking derivative with respect to $\bs{x}$ in \eqref{eq:defphie}, 
\begin{align*}
D_{x^i} \Phi^\e(s_0,w,\bs{x})&= -\frac{1}{2\e N}D_{\mu} \rho_*^2 (\mu^{\bs{x}},\mu_0)(x^i)\\
&=\frac{1}{2\e N}D_{\nu} \rho_*^2 (\mu^{\bs{x}},\mu_0)(x^i),
\end{align*}
where $D_{\nu} \rho_*^2$ denotes Lions' derivative of $\eta \mapsto \rho_*^2(\mu,\eta)$.  Taking derivative once more, 
\[ D^2_{x^i} \Phi^\e(s_0,w,\bs{x})=\frac{1}{2\e N}D^2_{x \nu} \rho_*^2 (\mu^{\bs{x}},\mu_0)(x^i)+\frac{1}{2\e N^2} D^2_{\nu} \rho_*^2(\mu^{\bs{x}_0},\mu_0)(x^i,x^i). \]
Again due to \eqref{eq:key}, $\mu_0$ is a maximizer of 
\[\mu \mapsto \Phi(t_0,z_0,\mu)-\frac{1}{2\e}\rho_*^2(\mu^{\bs{x}_0},\mu),\]
and first order condition yields
\[D_{\mu} \Phi(t_0,z_0,\mu_0)(x^i)= \frac{1}{2\e} D_{\nu} \rho^2_*(\mu^{\bs{x}_0},\mu)(x^i), \quad x^i \in \T^d,\] 
and hence 
\begin{align*}
D^2_{x\mu} \Phi(t_0,z_0,\mu_0)(x^i)= \frac{1}{2\e} D^2_{x\nu} \rho^2_*(\mu^{\bs{x}_0},\mu)(x^i), \quad x^i \in \T^d.
\end{align*}

The last claim is a property of super-convolution; see e.g. \cite[Lemma A.5]{usersguide}, and hence we finish proving the lemma. 
\end{proof}

Without loss of generality, we assume that  $V^N - \Phi^\e$ obtains a strict local minimal at $(s_0,w_0,\bs{x}_0)$. Denoting $\theta_0=(t_0,z_0,\mu_0)$, a direct application of Jensen's lemma (\cite[Lemma A.3]{usersguide})  and Lemma~\ref{lem:derivatives} shows that 
\[\left(\pa_t \Phi(\theta_0), D_{\mu} \Phi(\theta_0)(x^i_0), D_{x \mu}^2 \Phi(\theta_0)(x^i_0)+\frac{1}{2\e N} D^2_{\nu} \rho_*^2(\mu^{\bs{x}_0},\mu_0)(x^i_0,x^i_0), \Delta_z \Phi(\theta_0) \right). \]
approximate derivatives of $\Phi^{\e}$ at $(s_0,w_0,\bs{x}_0)$. We summarize this key result in the following proposition. 
\begin{proposition}\label{prop:jetclosure}
There exists a sequence of $\T^d \ni w_n \to w_0$ and $\R^d \ni p_n \to 0$ such that 
\begin{align*}
V^N(t_0,w,\bs{x}_0)-\Phi^{\e}(t_0,w,\bs{x}_0)+ \langle p_n, w \rangle \quad \text{obtains a local minimum at $w_n$},
\end{align*}
and $w \mapsto V^N(t_0,w,\bs{x}_0)-\Phi^{\e}(t_0,w,\bs{x}_0)$ is second-order differentiable at $w_n$. Moreover, for all $i=1,\dotso, N$, the limit point of 
\[\left( \pa_s \Phi^{\e},ND_{x^i}\Phi^{\e},ND^2_{x^i} \Phi^{\e},\Delta_w \Phi^{\e} \right)(s_0,w_n,\bs{x}_0), \]
as $n \to \infty$, is equal to 
\begin{align*}
 \left(\pa_t \Phi(\theta_0), D_{\mu} \Phi(\theta_0)(x^i_0), D_{x \mu}^2 \Phi(\theta_0)(x^i_0)+\frac{1}{2\e N} D^2_{\nu} \rho_*^2(\mu^{\bs{x}_0},\mu_0)(x^i_0,x^i_0), \Delta_z \Phi(\theta_0) \right).  
\end{align*}
\end{proposition}

As $v^N$ is a classical solution to \eqref{eq:nHJB}, it can be easily seen that $V^N$ is a classical solution to the following equation
\begin{equation*}
\begin{cases}
 -\pa_t V^N(t,w,\bs{x})=  \sum_{i=1}^N \frac{H^e(x^i,w, ND_{x^i}V^N(t,w,\bs{x}),\mu^{\bs{x}})}{N}+\sum_{i=1}^N \Delta_{x^i}V^N(t,w,\bs{x}) \\
\quad \quad \quad \quad \quad \quad \ \ \ \ \ \ +a \Delta_w V^N(t,w,\bs{x}),  \\
\quad \  V^N(T,w,\bs{x})=G^e(w,\mu^{\bs{x}}),
\end{cases}
\end{equation*}
and thus also a viscosity solution. Remember that $\theta_0=(t_0,z_0,\mu_0)$ and $V^N - \Phi^\e$ obtains a local minimal at $(s_0,w_0,\bs{x}_0)$ in \eqref{eq:Nmaxmia}. Setting $\e=\alpha(N)$ in Lemma~\ref{lem:maximizererror} below, it can be seen that $s_0 <T$ for large enough $N$. Therefore due to Proposition~\ref{prop:jetclosure} and the definition of viscosity solution
\begin{align*}
-\pa_t \Phi(\theta_0) \geq & \frac{1}{N}\sum_{i=1}^N H^e(x_0^i,w_0, D_{\mu} \Phi(\theta_0)(x^i_0),\mu^{\bs{x}_0}) \\
&+\sum_{i=1}^N \left( \frac{1}{N}\tr{D_{x\mu}^2 \Phi}(\theta_0,x_0^i)+\frac{1}{2\e N^2} D^2_{\nu} \rho_*^2(\mu^{\bs{x}_0},\mu_0)(x^i_0,x^i_0) \right)+a \Delta_z \Phi(\theta_0).
\end{align*}
Comparing the inequality above with \eqref{eq:supersol}, we define the error term $E(N,\e)$
\begin{align}\label{eq:error}
E(N,\e):=&\frac{1}{N}\sum_{i=1}^N  H^e(x_0^i,z_0, D_{\mu} \Phi(\theta_0)(x^i_0),\mu^{\bs{x}_0})  -  H^e(x_0^i,w_0, D_{\mu} \Phi(\theta_0)(x^i_0),\mu^{\bs{x}_0})  \notag \\
+&\int H^e(x,z_0,D_{\mu} \Phi(\theta_0)(x),\mu^{\bs{x}_0}) \, \mu_0(dx)-\sum_{i=1}^N \frac{H^e(x_0^i,z_0, D_{\mu} \Phi(\theta_0)(x^i_0),\mu^{\bs{x}_0})}{N}   \notag \\
+ &\int (H^e(x,z_0,D_{\mu} \Phi(\theta_0)(x),\mu_0) -H^e(x,z_0,D_{\mu} \Phi(\theta_0)(x),\mu^{\bs{x}_0})) \, \mu_0(dx)  \notag \\
+ &\int \tr (D_{x\mu}^2 \Phi)(\theta_0,x) \, (\mu_0-\mu^{\bs{x}_0})(dx) - \sum_{i=1}^N \frac{1}{2\e N^2} D^2_{\nu} \rho_*^2(\mu^{\bs{x}_0},\mu_0)(x^i_0,x^i_0).
\end{align}
It can be easily checked that \eqref{eq:supersol} is satisfied with such $E(N,\e)$. We will provide an upper bound of \eqref{eq:error}. To this end, let us first estimate $|s_0-t_0|$, $|z_0-w_0|$ and $\rho_*(\mu^{\bs{x}_0},\mu_0)$.

\begin{lemma}\label{lem:maximizererror}
We have the following estimates,
\begin{align*}
|t_0-s_0| \leq C \e^{2/3}, \quad |z_0-w_0| \leq C \e, \quad \rho_*(\mu_0,\mu^{\bs{x}_0}) \leq C \left(\e+\alpha(N) +\sqrt{\e \alpha(N)} \right),
\end{align*} 
where $C$ is a constant independent of $\e, N$.
\end{lemma}
\begin{proof}
According to Lemma~\ref{lem:regular}, we have 
\[\left| V^N(t,z,\bs{x})-V^N(s,w,\bs{x})    \right| \leq C (\sqrt{|t-s|}+|w-z|).     \]
Recall that 
\begin{align*}
\ol{V}^{N,\e}(t_0, z_0, \mu_0)=& V^N(s_0,w_0,\bs{x}_0)+\frac{1}{2\e}|s_0-t_0|^2+\frac{1}{2\e} |w_0-z_0|^2+ \frac{1}{2\e}\rho_*^2(\mu^{\bs{x}_0},\mu_0) \\
\leq &V^N(t_0,w_0,\bs{x}_0)+\frac{1}{2\e} |w_0-z_0|^2+ \frac{1}{2\e}\rho_*^2(\mu^{\bs{x}_0},\mu_0).
\end{align*}
Therefore we get inequalities 
\begin{align*}
\frac{1}{2\e}|s_0-t_0|^2 \leq V^N(t_0,w_0,\bs{x}_0)-V^N(s_0,w_0,\bs{x}_0) \leq C\sqrt{|s_0-t_0|},
\end{align*}
and hence $|s_0-t_0| \leq C \e^{2/3}$. By a similar argument, we obtain that $|w_0-z_0| \leq C \e$.

Thanks to Lemma~\ref{lem:dense}, take $\bs{x}' \in \mathbb{T}^{dN}$ such that $\rho_*(\mu^{\bs{x}'},\mu_0) \leq C\alpha(N)$. By the definition of $\ol{V}^{N,\e}$, 
\begin{align*}
& V^N(s_0,w_0,\bs{x}_0)+\frac{1}{2\e}|s_0-t_0|^2+\frac{1}{2\e} |w_0-z_0|^2+ \frac{1}{2\e}\rho_*^2(\mu^{\bs{x}_0},\mu_0) \\
& \leq V^N(s_0,w_0,\bs{x}')+\frac{1}{2\e}|s_0-t_0|^2+\frac{1}{2\e} |w_0-z_0|^2+ \frac{1}{2\e}\rho_*^2(\mu^{\bs{x}'},\mu_0).
\end{align*}
and therefore in conjunction with Proposition~\ref{prop:Lipschitz} we have
\begin{align*}
&\frac{1}{2\e}\rho_*^2(\mu^{\bs{x}_0},\mu_0)-\frac{1}{2\e} \rho_*^2(\mu^{\bs{x}'},\mu_0) \leq  V^N(s_0,w_0,\bs{x}')-V^N(s_0,w_0,\bs{x}_0) \\
&\leq C(\rho_*(\mu^{\bs{x}_0}, \mu^{\bs{x}'})+\alpha(N))\leq C (\rho_*(\mu^{\bs{x}_0},\mu_0)+2\alpha(N) ).
\end{align*}
Hence we obtain the inequality
\begin{align*}
\frac{1}{2\e}\rho_*^2(\mu^{\bs{x}_0},\mu_0)-C \rho_*(\mu^{\bs{x}_0},\mu_0) \leq \frac{1}{2\e} C^2\alpha(N)^2 + 2C \alpha(N),
\end{align*}
and thus 
\begin{align*}
\rho_*(\mu^{\bs{x}_0},\mu_0) \leq C \left(\e+\alpha(N) +\sqrt{\e \alpha(N)} \right),
\end{align*}
where $C$ is a constant independent of $\e,N$. 

\end{proof}

According to \cite[Lemma 5.1,5.4]{SoYa24}, we have 
\begin{align}
\rho_*^2(\mu,\nu) =&\sum_{l \in \mathbb Z^d} (1+|l|^2)^{-k_*} |F_l(\mu-\nu)|^2, \notag \\
D_{\nu} \rho_*^2(\mu,\nu)(x)=&-2i \sum_{l \in \mathbb Z^d} l(1+|l|^2)^{-k_*} F_l(\nu-\mu)e_l^*(x), \label{eq:derivative}  \\
D^2_{\nu} \rho_*^2(\mu,\nu)(x,y)=&-2\sum_{l \in \mathbb Z^d} l^2(1+|l|^2)^{-k_*} F_l(\nu-\mu)e_l^*(x)e_l(y). \notag
\end{align}
Therefore by Cauchy Schwarz inequality, 
\begin{align}\label{eq:nunubound}
\frac{1}{2} |D_{\nu} \rho_*^2(\mu,\nu)(x)| & \leq  \rho_*(\mu,\nu) \sqrt{\sum_{l \in \Z^d} |l|^2 (1+|l|^2)^{-k_*}}, \notag \\
\frac{1}{2}|D^2_{\nu} \rho_*^2(\mu,\nu) (x,y)| & \leq  \rho_*(\mu,\nu)\sqrt{\sum_{l \in \Z^d} |l|^4 (1+|l|^2)^{-k_*}}.
\end{align}
where the constants $\sqrt{\sum_{l \in \Z^d} |l|^2 (1+|l|^2)^{-k_*}}$ and $\sqrt{\sum_{l \in \Z^d} |l|^4 (1+|l|^2)^{-k_*}}$ are finite due to our choice of $k_*$. Then according to Lemma~\ref{lem:derivatives} and Lemma~\ref{lem:maximizererror}, 
\begin{align*}
\left| D_{\mu} \Phi(t_0,z_0,\mu_0)(x) \right| \leq \frac{\sqrt{\sum_{l \in \Z^d} |l|^2 (1+|l|^2)^{-k_*}}\left(\e+\alpha(N)+\sqrt{\e \alpha(N)}\right)}{\e }=:R(\e,N). 
\end{align*}
Here is an upper bound of $E(\e, N)$ depending on $R(\e,N)$. 
\begin{lemma}
We have that 
\begin{align*}
E(\e,N) \leq  & C_{R(\e,N)} \e+ C_{R(\e,N)}(1+1/(\e N))\rho_*(\mu^{\bs{x}_0},\mu_0)+\frac{1}{2 \e } \rho^2_{*}(\mu^{\bs{x}_0},\mu_0) \\
& +\frac{\e C_{R(\e,N)}^2}{4 }\left( \frac{\rho_*(\mu^{\bs{x}_0},\mu_0)}{\e }  + \frac{\rho^{k_*-1}_*(\mu^{\bs{x}_0},\mu_0)}{\e^{k_*-1}}\right)^2,
\end{align*}
where $C_{R(\e,N)}$ is a constant that only depends on $R(\e,N)$. 
\end{lemma} 
\begin{proof}

Let us denote the terms on the right-hand side of \eqref{eq:error} by $(I), (II), (III),$ and $(IV)$ from the top to the bottom, and we have the following estimate. 

\bigskip

\noindent \emph{Estimate of $(I)$:} Recall that $H^e(x,z,p,\mu):= H(x+z,p, (I_d+z)_{\#}\mu)$, and hence 
\begin{align*}
    |H^e(x,z,p,\mu)- H^e(x,w,p,\mu)| &\leq C_R ( |w-z| + W_1((I_d+z)_{\#}\mu,(I_d+w)_{\#}\mu )) \\
    &\leq 2C_{R(\e,N)} |w-z|.
\end{align*}
Together with Lemma~\ref{lem:maximizererror}, we conclude that 
$(I) \leq C_{R(\e,N)} \e $. 

\bigskip

\noindent \emph{Estimate of $(II)$:} Let us denote $f(x):=H^e(x,z_0,\pa_{\mu}\Phi(\theta_0)(x), \mu^{\bs{x}_0})$. Then it is straightforward that $(II)= \int f(x) \, (\mu_0-\mu^{\bs{x}_0})(dx)$. Due to Lemma~\ref{lem:derivatives}, 
\begin{align*}
D_{\mu}\Phi(\theta_0)(\cdot)= \frac{1}{2 \e } D_{\nu} \rho_*^2(\mu^{\bs{x}_0},\mu_0)(\cdot),
\end{align*}
According to \cite[Lemma 5.4]{SoYa24} the term on the right is in the Sobolev space $\mathbb H^{k_*-1}$. Indeed, according to \eqref{eq:derivative} 
\begin{align*}
\lVert D_{\nu} \rho_*^2(\mu^{\bs{x}_0},\mu_0)(\cdot)\rVert^2_{k_*-1} &=4 \sum_{l \in \Z^d} (1+|l|^2)^{k_*-1} |l|^2(1+|l|^2)^{-2k_*}|F_l(\nu-\mu)|^2 \\
&\leq 
4 \rho_*^2(\mu^{\bs{x}_0}, \mu_0), 
\end{align*} 
and hence 
\begin{align*}
\lVert D_{\mu}\Phi(\theta_0) \rVert_{{k_*-1}} \leq  \frac{1}{\e } \rho_*(\mu^{\bs{x}_0},\mu_0).
\end{align*}

As $\lVert H^e(\cdot , z_0, \cdot, \mu^{\bs{x}_0})\lVert_{C^{k_*}(\T^d \times B_{R(\e,N)})} \leq C_{R(\e,N)}$, $f$ is actually a composition $g\circ (I_d,h)$ of two functions $g$ and $(I_d,h)$ with $ g=H^e(\cdot , z_0, \cdot, \mu^{\bs{x}_0}) \in C^{k_*}(\T^d \times B_{R(\e,N)})$ and $h=D_{\mu}\Phi(\theta_0)(\cdot) \in \mathbb H^{k_*-1}(\T^d)$. Therefore according to the chain rule, \cite[Remark 2, Section 5.2]{RuSi96}, 
\begin{align*}
\lVert f \rVert_{{k_*-1}}= \lVert g\circ (I_d,h) \rVert_{{k_*-1}} &\leq C(R_{\e,N}) \left( \lVert D_{\mu}\Phi(\theta_0) \rVert_{{k_*-1}} +\lVert D_{\mu}\Phi(\theta_0) \rVert_{{k_*-1}}^{k_*-1}\right) \\
& \leq C(R_{\e,N}) \left( \frac{\rho_*(\mu^{\bs{x}_0},\mu_0)}{\e }  + \frac{\rho^{k_*-1}_*(\mu^{\bs{x}_0},\mu_0)}{\e^{k_*-1}}\right) 
\end{align*}
where $C(R_{\e,N})$ is a constant that only depends on $\lVert H^e(\cdot , z_0, \cdot, \mu^{\bs{x}_0})\lVert_{C^{k_*}(\T^d \times B_{R(\e,N)})} \leq C_{R(\e,N)}$. Therefore we obtain the estimate 
\begin{align*}
(II) &\leq C_{R(\e,N)} \rho_{1-k_*}(\mu^{\bs{x}_0},\mu_0) \left( \frac{\rho_*(\mu^{\bs{x}_0},\mu_0)}{\e }  + \frac{\rho^{k_*-1}_*(\mu^{\bs{x}_0},\mu_0)}{\e^{k_*-1}}\right) \\
& \leq \frac{1}{\e } \rho^2_{1-k_*}(\mu^{\bs{x}_0},\mu_0)+\frac{\e C_{R(\e,N)}^2}{4 }\left( \frac{\rho_*(\mu^{\bs{x}_0},\mu_0)}{\e }  + \frac{\rho^{k_*-1}_*(\mu^{\bs{x}_0},\mu_0)}{\e^{k_*-1}}\right)^2. 
\end{align*}

\bigskip

\noindent \emph{Estimate of $(III)$:} It is straightforward from Assumption~\ref{assume} that $(III)$ is bounded from above by $C_{R(\e,N)}\rho_*(\mu^{\bs{x}_0},\mu_0)$.

\bigskip

\noindent \emph{Estimate of $(IV)$:} According to Lemma~\ref{lem:derivatives}, we compute 
\begin{align*}
(IV)=&-\frac{1}{\e  } \sum_{l \in \Z^d} \frac{|l|^2}{(1+|l|^2)^{k_*}}F_l(\mu_0-\mu^{\bs{x}_0}) \int e^*_l(x) \, (\mu_0-\mu^{\bs{x}_0})(dx)\\
&-\sum_{i=1}^N \frac{1}{2\e N^2} D^2_{\nu} \rho_*^2(\mu^{\bs{x}_0},\mu_0)(x^i_0,x^i_0) \\
=&-\frac{1}{\e  } \sum_{l \in \Z^d} \frac{|l|^2}{(1+|l|^2)^{k_*}}|F_l(\mu_0-\mu^{\bs{x}_0})|^2-\sum_{i=1}^N \frac{1}{2\e N^2} D^2_{\nu} \rho_*^2(\mu^{\bs{x}_0},\mu_0)(x^i_0,x^i_0) \\
=&-\frac{1}{ \e } \rho^2_{1-k_*}(\mu^{\bs{x}_0},\mu_0)+\frac{1}{ \e } \rho^2_{*}(\mu^{\bs{x}_0},\mu_0)-\sum_{i=1}^N \frac{1}{2\e N^2} D^2_{\nu} \rho_*^2(\mu^{\bs{x}_0},\mu_0)(x^i_0,x^i_0) \\
\leq & -\frac{1}{ \e } \rho^2_{1-k_*}(\mu^{\bs{x}_0},\mu_0)+\frac{1}{ \e } \rho^2_{*}(\mu^{\bs{x}_0},\mu_0)+ \frac{C}{\e N} \rho_*(\mu^{\bs{x}_0},\mu_0),
\end{align*}
where we  apply the estimate \eqref{eq:nunubound} to get the last inequality. Summing up the estimates for $(I), (II), (III), (IV)$, we conclude the result.
\end{proof}

To finish the proof, we show that $\ol{V}^{\e,N}(T,z,\mu) \geq G^e(z,\mu)$ up to some error. Suppose that 
\[\ol{V}^{N,\e}(T,z,\mu)=V^N(s_T,w_T,\bs{x}_T)+ \frac{1}{2\e}|T-s_T|^2 +\frac{1}{2\e}|z-w_T|^2+\frac{1}{2\e} \rho_*^2(\mu^{\bs{x}_T},\mu).\]
By the same argument as in Lemma~\ref{lem:maximizererror}, it can be shown that
\begin{align*}
|T-s_T| \leq C \e^{2/3}, \quad |z-w_T| \leq C \e, \quad \rho_*(\mu_0,\mu^{\bs{x}_T}) \leq C \left(\e+\alpha(N) +\sqrt{\e \alpha(N)} \right),
\end{align*} 
where $C$ is a constant independent of $\e, N$. Then thanks to the regularity of $V^N$ from Lemma~\ref{lem:regular} and Proposition~\ref{prop:Lipschitz}, we get that 
\begin{align*}
\ol{V}^{N,\e}(T,z,\mu) &\geq V^N(T,z,\mu)-C ( \e^{1/3}+\e)- C\alpha(N)-C \left(\e+\alpha(N) +\sqrt{\e \alpha(N)} \right) \\
&\geq G^e(z,\mu)-C(\e^{1/3}+\e+\alpha(N)+\sqrt{\e \alpha(N)} ).
\end{align*}

Together with \eqref{eq:supersol}, $\ol{V}^{N,\e}(t,z,\mu)+ (T-t)E(\e,N)+C(\e^{1/3}+\e+\alpha(N)+\sqrt{\e \alpha(N)} )$ is a viscosity supersolution to \eqref{eq:HJB}. Then due to the comparison principle in Proposition~\ref{prop:existence}, we have that 
\begin{align*}
\ol{V}^{N,\e}(t,z,\mu)+ (T-t)E(\e,N)+C(\e^{1/3}+\e+\alpha(N)+\sqrt{\e \alpha(N)} ) \geq V(t,z,\mu). 
\end{align*}
Recalling that by the definition of inf-convolution, $V^N(t,z,\bs{x}) \geq \ol{V}^{N,\e}(t,z,\mu^{\bs{x}})$, and hence 
\begin{align*}
V^N(t,z,\bs{x}) \geq V(t,z,\mu^{\bs{x}})-TE(\e,N)-C(\e^{1/3}+\e+\alpha(N)+\sqrt{\e \alpha(N)} ).
\end{align*}
Choosing $\e=\alpha(N)$, it can be easily checked that $R(\e,N)$ is bounded independent of $N$, and therefore
\begin{align*}
TE(\e,N)+C(\e^{1/3}+\e+\alpha(N)+\sqrt{\e \alpha(N)} ) & \leq C \alpha^{1/3}(N). 
\end{align*}
So that we finish the proof of \eqref{eq:finequality}.

\section*{Acknowledgments}
We would like to thank Joe Jackson for his comments on the first version of the paper.

\bibliographystyle{siam}
\bibliography{ref.bib}

\end{document}